\documentclass[12pt]{amsart}
\usepackage{amsmath,amssymb}
\usepackage[mathcal]{eucal}
\usepackage[pdftex,bookmarks,colorlinks,breaklinks]{hyperref}

\hypersetup{
  pdftitle={Shadowing and Hyperbolicity for Endomorphisms of Locally Compact Groups},
  pdfauthor={Dekui Peng},
  pdfsubject={Topological dynamics and locally compact groups},
  pdfkeywords={shadowing property, Anosov endomorphism, Lie group, tidy subgroup}
}

\newtheorem{theorem}{Theorem}[section]
\newtheorem{maintheorem}{Theorem}

\newtheorem{corollary}[theorem]{Corollary}
\newtheorem{lemma}[theorem]{Lemma}
\newtheorem{proposition}[theorem]{Proposition}
\theoremstyle{definition}

\newtheorem{definition}[theorem]{Definition}
\numberwithin{equation}{section}

\theoremstyle{remark}
\newtheorem{remark}[theorem]{Remark}
\newtheorem{example}[theorem]{Example}

\DeclareMathOperator{\di}{d}

\begin{document}

\title[Shadowing and Hyperbolicity]{Shadowing and Hyperbolicity for Endomorphisms of Locally Compact Groups}

\author[D. Peng]{Dekui Peng}
 \address[D. Peng]
	{\hfill\break Institute of Mathematics,
		\hfill\break Nanjing Normal University, 210023,
	\hfill\break China}
\email{pengdk10@lzu.edu.cn}



\begin{abstract}
We study shadowing, with respect to the left uniformity, for continuous
endomorphisms of Lie groups and totally disconnected locally compact
groups. For Lie groups, an endomorphism has shadowing if and only if its
differential is hyperbolic, with zero eigenvalues allowed. This includes
singular maps outside the classical theory of Anosov endomorphisms. As
consequences, positively expansive Lie group endomorphisms are
topologically expanding, while on connected semisimple Lie groups the
shadowing endomorphisms are precisely the nilpotent ones. On compact
connected Lie groups, the nonsingular case agrees with classical Anosov
theory. In sharp contrast, every continuous endomorphism of an arbitrary
totally disconnected locally compact group has shadowing, without
compactness, metrizability or invertibility assumptions; the proof uses
Willis' tidy-above decomposition. Consequently, in this category
topological expansion and the topologically Anosov property reduce to
positive expansiveness and expansiveness, respectively. We also discuss
group shifts and revisit Aoki's dense-orbit compactness result without
assuming metrizability.
\end{abstract}

\maketitle

\section{Introduction}

The shadowing property, also known as the pseudo-orbit tracing property,
originated in hyperbolic dynamics. It asserts that every sufficiently
accurate approximate orbit is uniformly followed by a genuine orbit. For
Anosov diffeomorphisms this is one of the basic manifestations of uniform
hyperbolicity, and it plays a central role in stability theory. At the
purely topological level, shadowing is naturally paired with expansiveness:
a homeomorphism which is both expansive and has shadowing is commonly
called topologically Anosov. For homeomorphisms of compact metric spaces,
Walters' theorem that expansiveness together with shadowing implies
topological stability is a fundamental result in this direction
\cite{Walters}; see also
\cite{AokiHiraide,DasLeeRichesonWiseman}.

The non-invertible smooth theory developed along a related but distinct
line. Shub initiated the systematic study of expanding endomorphisms of
compact manifolds \cite{Shub1969}. Ma\~n\'e and Pugh studied stability for
endomorphisms \cite{ManePugh1975}, and Przytycki developed the theory of
Anosov endomorphisms and Axiom~A endomorphisms \cite{Przytycki1976}. In the
non-invertible setting the unstable direction may depend on the entire
prehistory of a point, so the natural extension, or inverse-limit system,
becomes an intrinsic part of the theory. The usual definition of an Anosov
endomorphism therefore assumes that the map is a local diffeomorphism and
encodes hyperbolicity either through a stable bundle and an expanding
quotient bundle or through stable--unstable splittings over the natural
extension.

These classical theories suggest two questions for group endomorphisms.
First, to what extent can shadowing be characterized directly by the
algebraic or infinitesimal structure of the map? Second, what remains of
hyperbolic dynamics when the endomorphism is singular and hence falls
outside the standard local-diffeomorphism framework? The first question is
particularly natural for Lie groups, where every endomorphism induces a
linear endomorphism of the Lie algebra. The second is unavoidable for group
endomorphisms, since the differential may have a non-trivial kernel even
when the resulting dynamics still has shadowing.

For group automorphisms, expansiveness has long been studied as a rigid
structural property. In the compact setting, early results go back to Wu
\cite{Wu1966}; for locally compact solvable groups and Lie groups, see Aoki
\cite{Aoki1979} and Morimoto \cite{Morimoto1981}. More recently, Shah gave
a systematic treatment for locally compact groups \cite{Shah2020}, while
Gl\"ockner and Raja related expansiveness in totally disconnected locally
compact groups to contraction groups, quotients and the nub
\cite{GlocknerRaja2017,WillisNub}.

Morimoto also studied the shadowing
property for group automorphisms. In particular, the connected compact Lie
group automorphism case is essentially classical: his work treats Euclidean
and toral automorphisms and gives related structural restrictions for
automorphisms of compact connected Lie groups \cite{Morimoto1981}. Aoki later
proved a hyperbolic-lifting criterion for solenoidal automorphisms
\cite[Theorem~2]{AokiSolenoidal1983}. Gromov subsequently gave a
closely related bounded-error, unique-shadowing result for hyperbolic
automorphisms of simply connected nilpotent Lie groups, together with an
expanding-semigroup version \cite[\S2.3]{Gromov2012}. His uniform large-scale
formulation is not identical to the left-uniform pseudo-orbit tracing
property used here, but it is an important precursor. The new scope claimed
for the Lie theorem below is the usual left-uniform shadowing
characterization on arbitrary Lie groups, including non-compact,
disconnected, non-invertible and possibly singular cases.

Shadowing behaves very differently in the totally disconnected locally
compact category. Aoki proved that every automorphism of a compact totally
disconnected metrizable group has the shadowing property
\cite[Application~2, pp.~170--172]{AokiSplitting1985}; this result is also recorded as Theorem~D in
\cite{AokiDense1985}. There are also prior endomorphism results which imply
shadowing in important profinite subclasses. For a profinite group of
type~(F), Reid constructs a neighbourhood basis of open subgroups $V$ with
$\alpha(V)\leq V$ for every continuous endomorphism $\alpha$
\cite[Definition~3.1 and Proposition~3.2]{Reid2014}; the elementary error
recursion then shows that $x_0$ shadows every sufficiently accurate
pseudo-orbit. Moreover, Wibmer proves that a positively expansive profinite
endomorphism is conjugate to a one-sided group shift of finite type
\cite[Proposition~2.1 and Lemma~3.3]{Wibmer2022}, so the standard finite-type
shadowing theorem applies. Compact open subgroups provide a zero-dimensional
local structure, and Willis' tidy subgroup theory supplies stable and
unstable correction mechanisms beyond these subclasses. We have not found
an earlier automatic-shadowing theorem at the full generality used below:
arbitrary totally disconnected locally compact groups and continuous,
possibly non-invertible endomorphisms. Thus shadowing is not a hyperbolicity
assumption in this category, but is automatic for every continuous
endomorphism.

The purpose of this paper is to make this contrast precise. We work with
the left uniformity of a topological group. For Lie groups, shadowing is
shown to be equivalent to hyperbolicity of the differential. When the
differential is nonsingular, this identifies the shadowing maps with the
classical Anosov endomorphisms; when it is singular, the same criterion
extends beyond the traditional theory. For totally disconnected locally
compact groups, every continuous endomorphism has shadowing. These two
results also clarify when the topological notions of expansion and Anosov
dynamics reduce simply to positive expansiveness and expansiveness.

\subsection{Shadowing and expansiveness}

All topological spaces considered in this paper are Hausdorff. By a
dynamical system $(X,f)$ we mean a topological space $X$ together with a
continuous map $f:X\to X$.

\begin{definition}
Let $(X,d)$ be a metric space, and let $f:X\to X$ be continuous. For
$\delta>0$, a finite or infinite sequence $(x_k)$ is called a
\emph{$\delta$-pseudo-orbit} if
\[
d(f(x_k),x_{k+1})<\delta
\]
whenever both $x_k$ and $x_{k+1}$ are defined. Given $\varepsilon>0$, we say
that $(x_k)$ is \emph{$\varepsilon$-shadowed} by a point $x\in X$, or by the
orbit of $x$, if
\[
d(f^k(x),x_k)<\varepsilon
\]
for every index $k$ for which $x_k$ is defined.

We say that the system $(X,f)$, or simply $f$, has \emph{shadowing} if for
every $\varepsilon>0$ there exists $\delta>0$ such that every one-sided
infinite $\delta$-pseudo-orbit is $\varepsilon$-shadowed by some point of
$X$. We say that $f$ has \emph{finite shadowing} if the same condition holds
for every finite $\delta$-pseudo-orbit. If $f$ is a homeomorphism, the
\emph{two-sided shadowing property} is defined analogously using
pseudo-orbits $(x_k)_{k\in\mathbb Z}$ and requiring
$d(f^k(x),x_k)<\varepsilon$ for every $k\in\mathbb Z$.
\end{definition}

This notion extends naturally to uniform spaces \cite{YanZeng}. Since every topological
group carries a canonical left uniformity, it makes sense to define
shadowing for self-maps of topological groups. In this paper we mainly
consider endomorphisms rather than arbitrary continuous self-maps.

Throughout the paper, an endomorphism is always assumed to be continuous,
and an automorphism is always assumed to be a topological automorphism. We
denote by $\operatorname{End}(G)$ and $\operatorname{Aut}(G)$ the monoid of
endomorphisms and the group of automorphisms of $G$, respectively.

\begin{definition}
Let $G$ be a topological group and let $\alpha\in\operatorname{End}(G)$.
For a neighbourhood $V$ of $e$, a finite or one-sided infinite sequence
$(x_k)$ in $G$ is called a \emph{$V$-pseudo-orbit} for $\alpha$ if
\[
\alpha(x_k)^{-1}x_{k+1}\in V
\]
whenever both $x_k$ and $x_{k+1}$ are defined.

Let $U$ be a neighbourhood of $e$. We say that a finite or one-sided
infinite sequence $(x_k)$ is \emph{$U$-shadowed} by a point $x\in G$, or by
the orbit of $x$, if
\[
\alpha^k(x)^{-1}x_k\in U
\]
for every index $k$ for which $x_k$ is defined.

We say that $\alpha$ has \emph{shadowing} if for every neighbourhood $U$ of
$e$ there exists a neighbourhood $V$ of $e$ such that every one-sided
infinite $V$-pseudo-orbit for $\alpha$ is $U$-shadowed by some point of
$G$.

If $\alpha\in\operatorname{Aut}(G)$, a sequence $(x_k)_{k\in\mathbb Z}$ is
called a \emph{two-sided $V$-pseudo-orbit} if
\[
\alpha(x_k)^{-1}x_{k+1}\in V
\qquad (k\in\mathbb Z).
\]
We say that $\alpha$ has \emph{two-sided shadowing} if, for every
neighbourhood $U$ of
$e$, there exists a neighbourhood $V$ of $e$ such that every two-sided
$V$-pseudo-orbit is $U$-shadowed by a point $x\in G$, in the sense that
\[
\alpha^k(x)^{-1}x_k\in U
\qquad (k\in\mathbb Z).
\]

We say that $\alpha$ has \emph{finite shadowing} if for every neighbourhood
$U$ of $e$ there exists a neighbourhood $V$ of $e$ such that every finite
$V$-pseudo-orbit for $\alpha$ is $U$-shadowed by some point of $G$.
\end{definition}

In topological groups, expansiveness has a particularly simple form, since
it is enough to compare points with the identity.

\begin{definition}
Let $G$ be a topological group.
\begin{enumerate}
\item An automorphism $\alpha\in\operatorname{Aut}(G)$ is called
\emph{expansive} if there exists a neighbourhood $U$ of $e$ such that
\[
\bigcap_{n\in\mathbb Z}\alpha^{-n}(U)=\{e\}.
\]
Such a neighbourhood $U$ is called an \emph{expansive neighbourhood} for
$\alpha$.

\item An endomorphism $\alpha\in\operatorname{End}(G)$ is called
\emph{positively expansive} if there exists a neighbourhood $U$ of $e$ such
that
\[
\bigcap_{n\geq 0}\alpha^{-n}(U)=\{e\}.
\]
Such a neighbourhood $U$ is called a \emph{positively expansive
neighbourhood} for $\alpha$.
\end{enumerate}
\end{definition}

It is worth noting that if a locally compact group admits an expansive
automorphism, or a positively expansive endomorphism, then it is
first-countable and hence metrizable by the Birkhoff--Kakutani metrization
theorem.

Expansive automorphisms of topological groups have been studied for a long
time. In the compact group setting, early results go back to Wu
\cite{Wu1966}; for locally compact solvable groups and Lie groups, see
Aoki~\cite{Aoki1979} and Morimoto~\cite{Morimoto1981}. More recently,
expansive automorphisms of locally compact groups were studied
systematically by Shah~\cite{Shah2020}. In the totally disconnected locally
compact setting, Gl\"ockner and Raja~\cite{GlocknerRaja2017,WillisNub} studied
expansive automorphisms in relation to contraction groups, quotients, and
the nub. From the symbolic-dynamical viewpoint, Kitchens and Schmidt
established a fundamental structure theory for expansive automorphisms of
compact groups \cite{Kitchens1987,KS1989}, and Kitchens later obtained a structural
description of expansive dynamics on locally compact groups
\cite{Kitchens2021}. These results are especially relevant in the totally
disconnected setting: as Aoki's theorem and
Theorem~\ref{thm:main-tdlc} below show,
shadowing is automatic there, whereas expansiveness is the genuinely
restrictive condition carrying the main structural information.

Following standard terminology in topological dynamics
\cite{AokiHiraide,DasLeeRichesonWiseman}, we recall the terminology combining
expansiveness and shadowing.

\begin{definition}
Let $(X,f)$ be a topological dynamical system endowed with a fixed
compatible uniformity. We call $f$ \emph{positively expansive} if there is
an entourage $E$ such that
\[
(f^n(x),f^n(y))\in E\quad\text{for every }n\geq 0
\]
implies $x=y$. If $f$ is a homeomorphism, it is called \emph{expansive} if
the analogous implication holds for every $n\in\mathbb Z$.
\begin{enumerate}
\item The system $(X,f)$ is called \emph{topologically expanding} if $f$ is
positively expansive and has the shadowing property.
\item If $f$ is a homeomorphism, then $(X,f)$ is called
\emph{topologically Anosov} if $f$ is expansive and has the two-sided
shadowing property.
\end{enumerate}
\end{definition}

Recall also the corresponding smooth notions. A $C^1$ diffeomorphism
$f:M\to M$ of a compact smooth manifold is an \emph{Anosov diffeomorphism}
if there is a continuous $Df$-invariant splitting
$TM=E^{\mathrm s}\oplus E^{\mathrm u}$ on which forward iterates contract
uniformly on $E^{\mathrm s}$ and backward iterates contract uniformly on
$E^{\mathrm u}$. In this paper, an \emph{Anosov endomorphism} is a $C^1$
local diffeomorphism $f:M\to M$ for which there are a continuous
$Df$-invariant subbundle $E^{\mathrm s}\subseteq TM$, constants $C>0$ and
$0<\lambda<1$, and a Riemannian metric such that, for every $n\geq 0$,
\[
\|Df^n v\|\leq C\lambda^n\|v\|
\quad(v\in E^{\mathrm s}),
\]
while the induced map on $TM/E^{\mathrm s}$ satisfies
\[
\|\overline{Df^n}\,\bar v\|\geq C^{-1}\lambda^{-n}\|\bar v\|.
\]
For local diffeomorphisms this quotient-bundle formulation corresponds to
the stable--unstable description over the natural extension, in which the
unstable space may depend on the chosen prehistory; see
\cite{ManePugh1975,Przytycki1976,AokiHiraide}.

\subsection{Main results}

Our first main result is the following complete characterization for Lie
groups. See Theorem~\ref{LieChar} below.

\begin{maintheorem}\label{thm:main-lie}
Let $G$ be a Lie group with Lie algebra $\mathfrak g$, and let
$\alpha:G\to G$ be a continuous endomorphism. Then $\alpha$ has shadowing if
and only if the differential
\[
\di\alpha:\mathfrak g\to\mathfrak g
\]
is hyperbolic.
\end{maintheorem}

Here a linear endomorphism is called \emph{hyperbolic} if its spectrum is
disjoint from the unit circle; zero eigenvalues are allowed. The proof has
two parts. Hyperbolicity implies shadowing by solving a nonlinear error
equation in local coordinates. Conversely, shadowing of $\alpha$ forces
$\di\alpha$ to have local finite shadowing near the origin, and the
finite-dimensional linear criterion then yields hyperbolicity.

Theorem~\ref{thm:main-lie} fits naturally into the classical theory of
Anosov maps but also
goes beyond it. If $\di\alpha$ is nonsingular, then $\alpha$ is a local
diffeomorphism, and on a connected Lie group it is a covering map. In this
case, on compact Lie groups, Theorem~\ref{thm:main-lie} identifies shadowing
with the usual
hyperbolicity condition for Anosov endomorphisms. If $\di\alpha$ is
singular, however, $\alpha$ is not a local diffeomorphism and hence lies
outside the standard theory of Anosov endomorphisms, while
Theorem~\ref{thm:main-lie} still applies.

For automorphisms, the result has a particularly clean interpretation. On
a compact Lie group, hyperbolicity of $\di\alpha$ is equivalent to the
classical Anosov splitting of the tangent bundle. Combining this with the
known infinitesimal characterization of expansiveness gives the equivalence
of classical Anosov dynamics, topological Anosov dynamics, expansiveness,
one-sided shadowing and two-sided shadowing. In the compact connected case,
this is essentially contained in Morimoto's work
\cite{Morimoto1981}; our formulation extends the picture to arbitrary Lie
groups and supplies the endomorphism and singular cases.

We also obtain a corresponding statement for positively expansive
endomorphisms.

\begin{maintheorem}\label{thm:main-positive-expansive}
Let $G$ be a Lie group, and let $\alpha:G\to G$ be a continuous
endomorphism. If $\alpha$ is positively expansive, then $\alpha$ has
shadowing. Consequently, $\alpha$ is topologically expanding if and only if
it is positively expansive. Moreover, if $G$ admits a positively expansive
endomorphism, then the identity component $G_0$ is nilpotent.
\end{maintheorem}

This result is proved below as
Theorem~\ref{thm:positive-expansive-shadowing}. The infinitesimal form of
positive expansiveness is stronger than
hyperbolicity: a Lie group endomorphism is positively expansive exactly
when every eigenvalue of its differential has modulus strictly larger than
$1$. The final assertion in
Theorem~\ref{thm:main-positive-expansive} uses a classical theorem of
Bourbaki:
a finite-dimensional Lie algebra admitting a hyperbolic automorphism is
nilpotent \cite{BourbakiLie}. In the connected case, the spectral
characterization of positive expansiveness used here is the cyclic-semigroup
specialization of Bhattacharya's criterion
\cite[Theorem~A and Proposition~2.4]{Bhattacharya2004}. In the compact Lie
case, positive expansiveness makes the endomorphism open, so automatic
shadowing also follows from Sakai's theorem
\cite[Theorem~1]{Sakai2003}; compare \cite[Lemma~9]{Sakai2010}. The compact
connected structural consequence has an earlier antecedent in Eisenberg's
theorem on positively expansive surjective endomorphisms
\cite{EisenbergPositive1966}. Accordingly,
Theorem~\ref{thm:main-positive-expansive} is presented as a consequence of
Theorem~\ref{thm:main-lie}, rather than as an independent priority claim; its
additional scope is the non-compact Lie setting.

The Lie group characterization also gives a clean description in the
semisimple case. We call an endomorphism $\alpha$ \emph{nilpotent} if some iterate
of $\alpha$ is the trivial endomorphism. For a connected Lie group, this is
equivalent to the nilpotency of $\di\alpha$. We prove that if $G$ is a
connected semisimple Lie group and $\alpha\in\operatorname{End}(G)$, then
$\alpha$ has shadowing if and only if $\alpha$ is nilpotent. Thus, in the
semisimple Lie setting, non-trivial recurrent group structure obstructs
shadowing for automorphic dynamics.

This conclusion is genuinely Lie-theoretic and does not extend to all
connected compact groups. Indeed, if $S$ is a non-abelian simple connected
compact Lie group, then the full shift on $S^{\mathbb Z}$ is a shadowing
automorphism of a connected semisimple compact group. This example shows
that infinite-dimensional compact groups may admit symbolic shadowing
phenomena which are invisible at the finite-dimensional Lie level.

Our second main theorem concerns totally disconnected locally compact
groups. Here the situation is strikingly different. The compact metrizable
automorphism case was proved by Aoki
\cite[Application~2, pp.~170--172]{AokiSplitting1985}; see also
\cite[Theorem~D]{AokiDense1985}. Our proof uses Willis' tidy subgroup theory for
endomorphisms of totally disconnected locally compact groups
\cite{Willis1994,Willis2001,WillisEndomorphisms}.

\begin{maintheorem}\label{thm:main-tdlc}
Let $G$ be a totally disconnected locally compact group, and let
$\alpha:G\to G$ be a continuous endomorphism. Then $\alpha$ has shadowing.
\end{maintheorem}

For compact metrizable groups and automorphisms,
Theorem~\ref{thm:main-tdlc} goes back to Aoki; in the compact automorphism
case, a standard countable-quotient argument also removes metrizability. The
theorem below passes beyond these compact automorphism cases to arbitrary
locally compact groups and arbitrary continuous endomorphisms, without a
countability assumption. As noted above, Reid's type~(F) theory and Wibmer's symbolic
model already imply important profinite endomorphism subclasses. We have not
located the full formulation in the earlier literature. Thus, for totally
disconnected locally
compact groups, shadowing is not an additional hyperbolicity assumption;
it is a structural consequence of compact open subgroups and the tidying
procedure. The tidy-above decomposition plays the role of a
zero-dimensional substitute for the stable--unstable correction mechanism
familiar from hyperbolic dynamics. In this setting, topological
hyperbolicity reduces to expansiveness: a continuous endomorphism is
topologically expanding if and only if it is positively expansive, and an
automorphism is topologically Anosov if and only if it is expansive.

The conclusion that topological Anosov dynamics reduces to
expansiveness should be viewed together with the structure theories of
Kitchens--Schmidt \cite{KS1989} and Kitchens \cite{Kitchens2021}. In the totally disconnected
locally compact category, shadowing contributes no further restriction;
the substantive dynamical structure is therefore encoded by
expansiveness itself.

Finally, we revisit Aoki's compactness argument. In the metrizable
setting, Aoki proved that topological mixing together with the pseudo-orbit
tracing property forces compactness
\cite[Proposition~3]{AokiDense1985}, and
that the existence of a dense orbit implies both topological mixing and the
pseudo-orbit tracing property \cite[Proposition~4]{AokiDense1985}.
Incorporating
the latter proposition weakens the hypothesis in our compactness
consequence from topological mixing to the existence of a dense orbit.
Moreover, Theorem~\ref{thm:main-tdlc} makes the shadowing assertion in Aoki's
Proposition~4 automatic, so the substantial part of its proof devoted to
establishing the pseudo-orbit tracing property is no longer needed. A
compact quotient argument then removes metrizability. Previts and Wu also
reproved Aoki's closely related Haar-ergodic compactness theorem using
Willis' structure theory \cite{PrevitsWu2003}. The compactness phenomenon is
therefore prior work; the point here is the automatic-shadowing input and
the explicit compact-quotient reduction and non-metrizable mixing argument.

\section{The Lie Case}\label{Lie}

In this section we study shadowing for endomorphisms of Lie groups. We use
the standard fact that every continuous homomorphism between
finite-dimensional Lie groups is smooth (indeed, analytic), so its
differential at the identity is well defined. The main goal is to prove
that shadowing is completely determined by this induced linear map on the
Lie algebra: an endomorphism $\alpha$ has shadowing if and only if
$\di\alpha$ is hyperbolic.

We begin with a general compactness observation. For locally compact
groups, finite shadowing is already equivalent to shadowing. This will be
useful later, since the converse direction of the Lie group characterization
is proved by passing from shadowing of $\alpha$ to local finite shadowing of
$\di\alpha$. The local compactness assumption is essential here, as the
following example shows.

\begin{lemma}\label{ShadFin}
Let $G$ be a locally compact group and $\alpha: G\to G$ be a continuous endomorphism. Then $\alpha$ has shadowing if and only if it has finite shadowing.
\end{lemma}
\begin{proof}
Only finite shadowing $\Rightarrow$ shadowing requires proof. Let $U_0$ be
a neighbourhood of $e$. By local compactness, there is a compact symmetric
neighbourhood $C$ of $e$ such that $C\subseteq U_0$. Choose a neighbourhood
$V$ of $e$ such that every finite $V$-pseudo-orbit is $C$-shadowed by some
point.

Let $(x_n)_{n\geq 0}$ be an infinite $V$-pseudo-orbit. For each $n\geq 0$,
put
\[
K_n=\bigcap_{i=0}^n\alpha^{-i}(x_iC).
\]
The choice of $V$ implies that $K_n$ is non-empty. Moreover, $K_n$ is a
closed subset of the compact set $x_0C$, and hence is compact. The sequence
$(K_n)_{n\geq 0}$ is decreasing, so
\[
K:=\bigcap_{n\geq 0}K_n\neq\varnothing.
\]
If $y\in K$, then $\alpha^i(y)\in x_iC$ for every $i\geq 0$, and therefore
$\alpha^i(y)^{-1}x_i\in C=C^{-1}\subseteq U_0$. Thus $y$ $U_0$-shadows
$(x_n)$, and $\alpha$ has shadowing.
\end{proof}

\begin{lemma}\label{lem:one-two-sided}
Let $G$ be a locally compact group and let
$\alpha\in\operatorname{Aut}(G)$. Then $\alpha$ has shadowing if and only if
it has two-sided shadowing.
\end{lemma}

\begin{proof}
Assume first that $\alpha$ has shadowing. Let $U$ be a neighbourhood of
$e$. Choose a compact symmetric neighbourhood $C$ of $e$ contained in $U$,
and let $V$ witness $C$-shadowing for $\alpha$.

Let $(x_k)_{k\in\mathbb Z}$ be a two-sided $V$-pseudo-orbit. For each
$n\geq 1$, extend the finite segment
$x_{-n},\ldots,x_n$ to a one-sided $V$-pseudo-orbit by following the true
orbit of $x_n$ after time $n$. Hence there exists $y_n\in G$ such that
\[
\alpha^j(y_n)^{-1}x_{-n+j}\in C
\qquad (0\leq j\leq 2n).
\]
Put $z_n=\alpha^n(y_n)$. Then
\[
\alpha^k(z_n)^{-1}x_k\in C
\qquad (-n\leq k\leq n).
\]
In particular, $z_n\in x_0C$, which is compact. Passing to a convergent
subnet, say $z_{n_i}\to z$, and fixing $k\in\mathbb Z$, we obtain by
continuity and closedness of $C$ that
$\alpha^k(z)^{-1}x_k\in C\subseteq U$. Thus $z$ $U$-shadows the given
two-sided pseudo-orbit.

Conversely, let $(x_k)_{k\geq 0}$ be a one-sided pseudo-orbit for an
automorphism. Defining $x_{-k}=\alpha^{-k}(x_0)$ for $k\geq 1$ extends it
to a two-sided pseudo-orbit without introducing any additional error.
Therefore two-sided shadowing implies shadowing.
\end{proof}

\begin{example}
Let $G=H(\mathbb C)$ be the Fréchet space of entire functions endowed with
the compact-open topology, regarded as an abelian Polish group under addition.
For $\lambda\in\mathbb C$ with $|\lambda|>1$, define
\[
M_\lambda(f)=\lambda f,\qquad f\in H(\mathbb C).
\]
Then $M_\lambda$ is a continuous automorphism of the Polish group $G$, with
inverse $M_{\lambda^{-1}}$. By Bernardes and Peris \cite{BernardesPeris}, $M_\lambda$ has the
finite shadowing property but does not have the shadowing property.
\end{example}

We now turn to Lie groups. In classical hyperbolic dynamics, a hyperbolic
linear map is usually understood to be a linear automorphism whose spectrum
is disjoint from the unit circle. This convention is adapted to
diffeomorphisms and, more generally, to locally invertible dynamics.
However, the Lie group endomorphisms considered here need not be local
diffeomorphisms, and their differentials may be singular. We therefore
extend the usual terminology to arbitrary finite-dimensional linear
endomorphisms by retaining the same spectral condition and allowing zero
as an eigenvalue. Thus, in the invertible case, our definition agrees with
the standard one.

\begin{definition}
Let $E$ be a finite-dimensional real vector space, and let $T:E\to E$ be a
linear endomorphism. We say that $T$ is \emph{hyperbolic} if
\[
\sigma(T)\cap S^1=\varnothing,
\]
where $\sigma(T)$ is the spectrum of the complexification of $T$ and
$S^1=\{z\in\mathbb C:|z|=1\}$. Equivalently, every eigenvalue of $T$ has
modulus either strictly smaller than $1$ or strictly larger than $1$.
In particular, zero is allowed as an eigenvalue.
\end{definition}

The next lemma is the technical core of the sufficiency direction. It says that, when the differential is hyperbolic, every sufficiently small sequence of local errors can be corrected by a bounded sequence near the identity.

\begin{lemma}\label{lem:local-error-equation}
Let $G$ be a Lie group with Lie algebra $\mathfrak g$, let $\alpha:G\to G$ be a continuous endomorphism,
and put $A=\di\alpha:\mathfrak g\to\mathfrak g$. If $A$ is hyperbolic, then
for every neighbourhood $U$ of $e$ in $G$, there exists a neighbourhood
$V$ of $e$ such that, for every sequence $(e_n)_{n\geq 0}$ in $V$, the
equation
\begin{equation}\label{eq:error-equation}
d_{n+1}=\alpha(d_n)e_n,\qquad n\geq 0,
\end{equation}
has a solution $(d_n)_{n\geq 0}$ contained in $U$.
\end{lemma}

\begin{proof}
Choose an exponential chart $\theta:O\to E$ from a neighbourhood $O$ of
$e$ in $G$ onto a neighbourhood of $0$ in $E=\mathfrak g$, with
$\theta(e)=0$ and $(D\theta)(e)=\operatorname{id}_{\mathfrak g}$. We shall
shrink $O$ several times below.

In these coordinates, put
\[
\mathcal D=\{(\xi,a)\in\theta(O)\times\theta(O):
\alpha(\theta^{-1}(\xi))\theta^{-1}(a)\in O\}.
\]
Then $\mathcal D$ is an open neighbourhood of $(0,0)$. Define the smooth
map $F:\mathcal D\to E$ by
\[
F(\xi,a)=\theta\bigl(\alpha(\theta^{-1}(\xi))\theta^{-1}(a)\bigr),
\]
so that $F(0,0)=0$ and $D_\xi F(0,0)=A$. Write
\[
F(\xi,a)=A\xi+\psi(\xi,a).
\]
Thus $\psi(0,0)=0$ and $D_\xi\psi(0,0)=0$.

Since $A$ is hyperbolic, there is an $A$-invariant decomposition
\begin{equation}\label{eq:splitting}
E=E^{\mathrm{s}}\oplus E^{\mathrm{u}},
\end{equation}
where the spectrum of $A_{\mathrm{s}}=A|_{E^{\mathrm{s}}}$ lies inside the
open unit disk in $\mathbb{C}$, and $A_{\mathrm{u}}=A|_{E^{\mathrm{u}}}$ is invertible with
spectrum outside the closed unit disk. Choose norms on
$E^{\mathrm{s}}$ and $E^{\mathrm{u}}$, and use the product norm
\[
\|\xi\|=\max\{\|\xi^{\mathrm{s}}\|,\|\xi^{\mathrm{u}}\|\}
\]
on $E$, so that, for some $0<\lambda<1$,
\begin{equation}\label{eq:lambda-estimate}
\|A_{\mathrm{s}}\|\leq\lambda,\qquad
\|A_{\mathrm{u}}^{-1}\|\leq\lambda.
\end{equation}

We next choose the constants used in the contraction argument. Since
$D_\xi\psi(0,0)=0$ and $D_\xi\psi$ is continuous as a function of the pair
$(\xi,a)$, we may make the Lipschitz constant of $\psi$ in the first
variable arbitrarily small, uniformly for all sufficiently small values of
the second variable. More precisely, choose $L>0$ with
\begin{equation}\label{eq:L-small}
\frac{L}{1-\lambda}<\frac14.
\end{equation}
Then choose $\rho>0$ so small that
\[
\overline{B_\rho(0)}\times\overline{B_\rho(0)}\subseteq\mathcal D,
\qquad
\theta^{-1}(\overline{B_\rho(0)})\subseteq U,
\]
and
\[
\|D_\xi\psi(\xi,a)\|\leq L
\]
whenever $\|\xi\|\leq\rho$ and $\|a\|\leq\rho$. Since the closed ball is
convex, the mean value estimate gives, for
$\|\xi\|,\|\zeta\|\leq\rho$ and $\|a\|\leq\rho$,
\begin{equation}\label{eq:psi-lipschitz}
\|\psi(\xi,a)-\psi(\zeta,a)\|\leq L\|\xi-\zeta\|.
\end{equation}

Since $\psi(0,0)=0$, choose $0<\eta<\rho$ such that
\begin{equation}\label{eq:psi-zero-small}
\|\psi(0,a)\|\leq \frac{(1-\lambda)\rho}{4}
\end{equation}
whenever $\|a\|\leq\eta$. Combining \eqref{eq:L-small},
\eqref{eq:psi-lipschitz},  and
\eqref{eq:psi-zero-small}, we obtain, for all $\|\xi\|\leq\rho$ and
$\|a\|\leq\eta$,
\begin{equation}\label{eq:psi-bound}
\|\psi(\xi,a)\|
\leq
\|\psi(\xi,a)-\psi(0,a)\|+\|\psi(0,a)\|
\leq
L\rho+\frac{(1-\lambda)\rho}{4}
\leq
\frac{(1-\lambda)\rho}{2}.
\end{equation}

Let $V=\theta^{-1}(B_\eta(0))$. Now let $(e_n)_{n\geq 0}$ be a sequence
in $V$, and put $a_n=\theta(e_n)$. We shall solve
\begin{equation}\label{eq:coordinate-recurrence}
\xi_{n+1}=A\xi_n+\psi(\xi_n,a_n)
\end{equation}
with $\|\xi_n\|\leq\rho$ for all $n\geq 0$. Then
$d_n=\theta^{-1}(\xi_n)$ will solve \eqref{eq:error-equation} and will
belong to $U$.

Let $\mathcal X$ be the complete metric space of all bounded sequences
$\xi=(\xi_n)_{n\geq 0}$ in $E$ satisfying $\|\xi\|_\infty\leq\rho$. For
$\xi\in\mathcal X$, write
\[
\psi(\xi_n,a_n)
=
\psi^{\mathrm{s}}(\xi_n,a_n)+\psi^{\mathrm{u}}(\xi_n,a_n)
\]
according to the decomposition \eqref{eq:splitting}. Define
$\Gamma:\mathcal X\to \ell^\infty(E)$ by
\begin{align}
(\Gamma\xi)_n^{\mathrm{s}}
&=
\sum_{k=0}^{n-1}A_{\mathrm{s}}^{\,n-1-k}
\psi^{\mathrm{s}}(\xi_k,a_k),
\label{eq:Gamma-stable}\\
(\Gamma\xi)_n^{\mathrm{u}}
&=
-\sum_{k=n}^{\infty}A_{\mathrm{u}}^{\,n-1-k}
\psi^{\mathrm{u}}(\xi_k,a_k).
\label{eq:Gamma-unstable}
\end{align}
The second series converges because
$\|A_{\mathrm{u}}^{-1}\|\leq\lambda<1$.

We now check that $\Gamma$ maps $\mathcal X$ into itself. Put
\[
M_\xi=\sup_{n\geq 0}\|\psi(\xi_n,a_n)\|.
\]
By \eqref{eq:psi-bound}, $M_\xi\leq (1-\lambda)\rho/2$. For the stable
component, using \eqref{eq:lambda-estimate}, we have
\[
\|(\Gamma\xi)_n^{\mathrm{s}}\|
\leq
\sum_{k=0}^{n-1}\lambda^{n-1-k}M_\xi
\leq
\frac{M_\xi}{1-\lambda}.
\]
For the unstable component,
\[
\|(\Gamma\xi)_n^{\mathrm{u}}\|
\leq
\sum_{k=n}^{\infty}\lambda^{k+1-n}M_\xi
\leq
\frac{M_\xi}{1-\lambda}.
\]
Since the norm on $E$ is the maximum of the stable and unstable norms, it
follows that
\[
\|\Gamma\xi\|_\infty
\leq
\frac{M_\xi}{1-\lambda}
\leq
\frac{\rho}{2}
<\rho.
\]
Thus $\Gamma(\mathcal X)\subseteq\mathcal X$.

Next we show that $\Gamma$ is a contraction. Let $\xi,\zeta\in\mathcal X$
and put
\[
M_{\xi,\zeta}
=
\sup_{n\geq 0}
\|\psi(\xi_n,a_n)-\psi(\zeta_n,a_n)\|.
\]
By \eqref{eq:psi-lipschitz},
$M_{\xi,\zeta}\leq L\|\xi-\zeta\|_\infty$. Again, the estimates in \eqref{eq:lambda-estimate} give
\[
\|(\Gamma\xi-\Gamma\zeta)_n^{\mathrm{s}}\|
\leq
\frac{M_{\xi,\zeta}}{1-\lambda},
\qquad
\|(\Gamma\xi-\Gamma\zeta)_n^{\mathrm{u}}\|
\leq
\frac{M_{\xi,\zeta}}{1-\lambda}.
\]
Therefore,
\[
\|\Gamma\xi-\Gamma\zeta\|_\infty
\leq
\frac{L}{1-\lambda}\|\xi-\zeta\|_\infty
<
\frac14\|\xi-\zeta\|_\infty.
\]
Hence $\Gamma$ is a contraction on $\mathcal X$. By the Banach fixed point
theorem, $\Gamma$ has a fixed point $\xi=(\xi_n)_{n\geq 0}$ in $\mathcal X$.

It remains to check that this fixed point solves
\eqref{eq:coordinate-recurrence}. For the stable component, the definition
\eqref{eq:Gamma-stable} gives
\[
\xi_{n+1}^{\mathrm{s}}
=
A_{\mathrm{s}}\xi_n^{\mathrm{s}}
+
\psi^{\mathrm{s}}(\xi_n,a_n).
\]
For the unstable component, shifting the series in
\eqref{eq:Gamma-unstable} gives
\[
\xi_{n+1}^{\mathrm{u}}
=
A_{\mathrm{u}}\xi_n^{\mathrm{u}}
+
\psi^{\mathrm{u}}(\xi_n,a_n).
\]
Thus
\[
\xi_{n+1}=A\xi_n+\psi(\xi_n,a_n)=F(\xi_n,a_n).
\]
Consequently, with $d_n=\theta^{-1}(\xi_n)$, we have
\[
d_{n+1}=\alpha(d_n)e_n
\]
for all $n\geq 0$. Since $\|\xi_n\|\leq\rho$ and
$\theta^{-1}(\overline{B_\rho(0)})\subseteq U$, the sequence
$(d_n)_{n\geq 0}$ is
contained in $U$. This completes the proof.
\end{proof}

We now use the local error equation to trace pseudo-orbits of the group endomorphism.

\begin{proposition}\label{HypSha}
Let $G$ be a Lie group, and let $\alpha:G\to G$ be a continuous
endomorphism. If $\di \alpha$ is hyperbolic, then $\alpha$ has shadowing.
\end{proposition}

\begin{proof}
Let $U$ be a neighbourhood of $e$ in $G$. By Lemma~\ref{lem:local-error-equation},
choose a neighbourhood $V$ of $e$ such that every sequence $(e_n)$ in $V$
admits a solution $(d_n)$ in $U$ of
$
d_{n+1}=\alpha(d_n)e_n.
$

Let $(x_n)_{n\geq 0}$ be a $V$-pseudo-orbit for $\alpha$. Write
$
x_{n+1}=\alpha(x_n)e_n
$
with $e_n\in V$. Choose a solution $(d_n)$ in $U$ of the above 
equation, and put
$
y=x_0d_0^{-1}.
$
We claim that
\[
x_n=\alpha^n(y)d_n
\]
for all $n\geq 0$. This is true for $n=0$. If it holds for $n$, then
\[
x_{n+1}
=
\alpha(x_n)e_n
=
\alpha(\alpha^n(y)d_n)e_n
=
\alpha^{n+1}(y)\alpha(d_n)e_n
=
\alpha^{n+1}(y)d_{n+1}.
\]
Hence the claim follows by induction. Therefore
\[
\alpha^n(y)^{-1}x_n=d_n\in U
\]
for every $n\geq 0$. Thus $(x_n)$ is $U$-shadowed by $y$.
In other words, $\alpha$ has shadowing.
\end{proof}

For the converse implication, we shall need a linear shadowing
criterion. For finite-dimensional linear automorphisms, the equivalence
between hyperbolicity and shadowing is classical and goes back at least
to Ombach \cite{Ombach1994}. In fact, Bernardes, Cirilo, Darji,
Messaoudi and Pujals proved the corresponding one-sided result for
arbitrary, not necessarily invertible, finite-dimensional linear
operators \cite[Lemma~45]{BernardesEtAl2018}. Moreover, shadowing and
finite shadowing coincide for every bounded linear operator on a Banach
space \cite{BernardesPeris}. The additional condition needed in our
argument is the following localized finite version, which allows us to
pass from a Lie group endomorphism to its differential using only a
neighbourhood of the identity. The precise localized implication needed here does not seem to have
been recorded in this form, so we include the argument.

For a linear endomorphism $A:E\to E$, we say that $A$ has
\emph{local finite shadowing at the origin} if there exists a
neighbourhood $W$ of $0$ such that, for every $\varepsilon>0$, there
exists $\delta>0$ such that every finite $\delta$-pseudo-orbit of $A$
contained in $W$ is $\varepsilon$-shadowed by some point of $E$. Thus finite
shadowing implies local finite shadowing at the origin.

\begin{lemma}\label{LineShadEqu}
Let $E$ be a finite-dimensional real normed vector space, and let
$A:E\to E$ be a linear endomorphism. Then the following conditions are
equivalent:
\begin{enumerate}
\item $A$ has shadowing;
\item $A$ has finite shadowing;
\item $A$ has local finite shadowing at the origin;
\item $A$ is hyperbolic.
\end{enumerate}
\end{lemma}

\begin{proof}
Proposition~\ref{HypSha}, applied to the additive Lie group $E$, gives
$(4)\Rightarrow(1)$, while $(1)\Rightarrow(2)\Rightarrow(3)$ is immediate.
It remains to prove $(3)\Rightarrow(4)$.

Suppose that $A$ has local finite shadowing.
We shall show that $A$ is hyperbolic. Arguing by contradiction, assume that
$A$ is not hyperbolic. Then the complexification $A_{\mathbb C}$ has an
eigenvalue $\lambda$ with $|\lambda|=1$.

Choose a Jordan block $B$ for $\lambda$ of size $q\geq 1$. Thus, in a suitable
basis $v_0,\ldots,v_{q-1}$ of this block, we have
\[
A_{\mathbb C}v_0=\lambda v_0,\qquad
A_{\mathbb C}v_j=\lambda v_j+v_{j-1}\quad (1\leq j<q).
\]
Extend this to a Jordan basis of $E_{\mathbb C}$. Define a complex linear
functional
\[
\ell:E_{\mathbb C}\to\mathbb C
\]
as follows: if
\[
Z=\sum_{j=0}^{q-1} z_jv_j+Z',
\]
where $Z'$ belongs to the span of all the other Jordan basis vectors, then
$
\ell(Z)=z_0.
$
Thus $\ell$ extracts the coefficient of $v_0$ in the chosen Jordan block
and vanishes on all other Jordan blocks.

We claim that, for every $Y\in E_{\mathbb C}$, the function
\[
i\mapsto \lambda^{-i}\ell(A_{\mathbb C}^iY)
\]
is the restriction to $\mathbb N$ of a polynomial in $i$ of degree at most
$q-1$. Indeed, write the component of $Y$ in $B$ as
\[
Y_B=\sum_{j=0}^{q-1} y_jv_j.
\]
Only this component contributes to $\ell(A_{\mathbb C}^iY)$. On $B$ we have
\[
A_{\mathbb C}=\lambda I+N,
\]
where
\[
Nv_0=0,\qquad Nv_j=v_{j-1}\quad (1\leq j<q).
\]
Hence
\[
A_{\mathbb C}^i
=
(\lambda I+N)^i
=
\sum_{k=0}^{q-1}\binom{i}{k}\lambda^{i-k}N^k
\]
on $B$, since $N^q=0$. Therefore
\[
\lambda^{-i}\ell(A_{\mathbb C}^iY)
=
\sum_{k=0}^{q-1}\binom{i}{k}\lambda^{-k}\ell(N^kY_B).
\]
The right-hand side is a polynomial in $i$ of degree at most $q-1$, because
each $\binom{i}{k}$ is a polynomial in $i$ of degree $k$.

Let $\varphi(t)=t^q$ on $[0,1]$. The space of complex polynomials of degree
at most $q-1$ is a finite-dimensional, hence closed, subspace of
$C([0,1],\mathbb C)$, and $\varphi$ does not belong to it. Consequently,
there exists $c>0$ such that
\begin{equation}\label{eq:distance-from-low-degree-polynomials}
\sup_{0\leq t\leq 1}|\varphi(t)-p(t)|\geq c
\end{equation}
for every complex polynomial $p$ of degree at most $q-1$. We shall use the following
discrete consequence. After decreasing $c$ if necessary, there exists
$N_0\geq 1$ such that, for every $N\geq N_0$ and every complex polynomial $p$ of
degree at most $q-1$,
\begin{equation}\label{eq:discrete-polynomial-separation}
\max_{0\leq i\leq N}|\varphi(i/N)-p(i/N)|\geq c.
\end{equation}
Indeed, if this failed, then for a sequence $N_j\to\infty$ one could find
polynomials $p_j$ of degree at most $q-1$ such that
\[
\max_{0\leq i\leq N_j}|\varphi(i/N_j)-p_j(i/N_j)|\to 0.
\]
The polynomials $p_j$ are bounded at $q$ separated grid points, for all
large $j$, hence their coefficients are bounded. Passing to a subsequence,
$p_j$ converges uniformly on $[0,1]$ to a polynomial $p$ of degree at most
$q-1$. Since the grids become dense in $[0,1]$, we get $p=\varphi$ on
$[0,1]$, contradicting \eqref{eq:distance-from-low-degree-polynomials}.

Let $B_\rho(0)$ be a ball on which local finite shadowing is assumed to hold. We shall construct, inside $B_\rho(0)$, arbitrarily fine
finite pseudo-orbits which cannot be shadowed with a fixed positive
accuracy.

Choose $\tau>0$ small enough so that the real vectors constructed below all
belong to $B_\rho(0)$. For $N\geq N_0$, put
\[
Z_i=\tau\lambda^i\varphi(i/N)v_0,\qquad 0\leq i\leq N.
\]

If $\lambda$ is non-real, then the conjugate block for $\overline{\lambda}$
also appears in $E_{\mathbb C}$. Our functional $\ell$ is chosen to vanish
on that conjugate block. Hence, for
\[
X_i=2\operatorname{Re} Z_i=Z_i+\overline{Z_i},
\]
we have
\[
\ell(X_i)
=
\lambda^i\tau\varphi(i/N).
\]
If $\lambda$ is real, then $\lambda=\pm1$ and the Jordan chain may be
chosen in $E$. In this case let $X_i=Z_i$ and
\[
\ell(X_i)=\lambda^i\tau\varphi(i/N).
\]
Thus, in both cases, we have
\begin{equation}\label{eq:ell-Xi}
\ell(X_i)=\lambda^i\tau\varphi(i/N)
\qquad (0\leq i\leq N).
\end{equation}

We next check that $(X_i)_{i=0}^N$ is an arbitrarily fine pseudo-orbit as
$N\to\infty$. Since $v_0$ is an eigenvector with eigenvalue $\lambda$,
\[
Z_{i+1}-A_{\mathbb C}Z_i
=
\tau\lambda^{i+1}
\bigl(\varphi((i+1)/N)-\varphi(i/N)\bigr)v_0.
\]
Hence, by uniform continuity of $\varphi$,
\[
\max_{0\leq i<N}\|Z_{i+1}-A_{\mathbb C}Z_i\|\to 0
\quad\text{as }N\to\infty.
\]
Taking real parts in the non-real case, we also get
\begin{equation}\label{eq:Xi-pseudo-small}
\max_{0\leq i<N}\|X_{i+1}-AX_i\|\to 0
\quad\text{as }N\to\infty.
\end{equation}

We now show that these finite pseudo-orbits cannot be shadowed with a fixed
positive accuracy. Let $C>0$ be such that
\[
|\ell(W)|\leq C\|W\|
\qquad (W\in E_{\mathbb C}).
\]
Choose $\varepsilon_0>0$ with $C\varepsilon_0<\tau c$.
Suppose that there is some $Y\in E$ $\varepsilon_0$-shadowing
$X_0,\ldots,X_N$, that is,
\[
\|A^iY-X_i\|<\varepsilon_0
\qquad (0\leq i\leq N).
\]
Applying $\ell$ after complexification and using \eqref{eq:ell-Xi}, we get
\[
\left|
\lambda^{-i}\ell(A_{\mathbb C}^iY)-\tau\varphi(i/N)
\right|
=
\left|
\ell(A_{\mathbb C}^iY)-\lambda^i\tau\varphi(i/N)
\right|
\leq C\varepsilon_0
\qquad (0\leq i\leq N).
\]
Recall that the function
\[
i\mapsto \lambda^{-i}\ell(A_{\mathbb C}^iY)
\]
is the restriction of a polynomial in $i$ of degree at most $q-1$. Hence
there is a polynomial $p$ of degree at most $q-1$ in the variable $t$ such
that
\[
p(i/N)=\frac{1}{\tau}\lambda^{-i}\ell(A_{\mathbb C}^iY)
\qquad (0\leq i\leq N).
\]
Therefore
\[
|\varphi(i/N)-p(i/N)|
\leq
\frac{C\varepsilon_0}{\tau}
<c
\qquad (0\leq i\leq N),
\]
contradicting \eqref{eq:discrete-polynomial-separation}. Thus no point
$\varepsilon_0$-shadows $X_0,\ldots,X_N$.

Now let $\delta>0$ be arbitrary. By \eqref{eq:Xi-pseudo-small}, for all
sufficiently large $N$ we have
\[
\|X_{i+1}-AX_i\|<\delta
\qquad (0\leq i<N).
\]
Thus, inside $B_\rho(0)$, we have constructed a finite $\delta$-pseudo-orbit
which is not $\varepsilon_0$-shadowed by any point. This contradicts
local finite shadowing. Hence $A$ must be hyperbolic.
\end{proof}

We can now prove the Lie group characterization. The sufficiency has already been obtained in Proposition~\ref{HypSha}. The converse is obtained by showing that shadowing of $\alpha$ forces the differential to have local finite shadowing.

\begin{theorem}\label{LieChar}
Let $G$ be a Lie group with Lie algebra $\mathfrak g$, and let $\alpha:G\to G$ be a continuous endomorphism.
Then $\alpha$ has shadowing if and only if
$\di\alpha:\mathfrak g\to\mathfrak g$ is hyperbolic.
\end{theorem}

\begin{proof}
The sufficiency is exactly Proposition~\ref{HypSha}. We prove the converse.
Assume that $\alpha$ has shadowing and put
$
A=\di\alpha.
$
By Lemma~\ref{LineShadEqu}, it is enough to show that $A$ has local finite
shadowing.

Choose
$r>0$ so small that the exponential map is injective on an open
neighbourhood containing both $\overline{B_r(0)}$ and
$A(\overline{B_r(0)})$. Fix $0<\varrho<r$. Given $\varepsilon>0$,
Replacing $\varepsilon$ by a smaller positive number if necessary, we may
assume that $\varepsilon<r$.

Set
$
C=\exp(\overline{B_\varrho(0)}).
$
Then $C$ is compact and $C\subseteq \exp(B_r(0))$. Hence there exists a
symmetric neighbourhood $W$ of $e$ in $G$ such that
\begin{equation}\label{eq:nbd}
CW\subseteq \exp(B_r(0)).
\end{equation}
The set
\[
D_\varepsilon=\{\exp(P)^{-1}\exp(Q):P,Q\in\overline{B_r(0)},
\ \|P-Q\|\geq\varepsilon\}
\]
is compact and does not contain $e$, because $\exp$ is injective on a
neighbourhood of $\overline{B_r(0)}$. Choose a symmetric neighbourhood
$O\subseteq W$ of $e$ such that $O\cap D_\varepsilon=\varnothing$. Then,
whenever $P,Q\in\overline{B_r(0)}$,
\begin{equation}\label{eq:uni}
\exp(P)^{-1}\exp(Q)\in O
\qquad \Longrightarrow
\qquad
\|P-Q\|<\varepsilon.
\end{equation}

Let $N$ be a neighbourhood of $e$ witnessing $O$-shadowing for $\alpha$;
that is, every $N$-pseudo-orbit is $O$-shadowed by some $y\in G$. Shrinking
$N$ if necessary, we may assume that $N\subseteq W$.

By continuity of the map
\[
(P,Q)\mapsto \exp(AP)^{-1}\exp(Q),
\]
and compactness of $\overline{B_\varrho(0)}$, there exists $\delta>0$ such
that, whenever $P,Q\in\overline{B_\varrho(0)}$ satisfy
\begin{equation}\label{eq:lin2gr}
\|Q-AP\|<\delta,
\end{equation}
one has $\exp(AP)^{-1}\exp(Q)\in N$.

Now let $X_0,\ldots,X_m\in B_\varrho(0)$ be a finite $\delta$-pseudo-orbit
for $A$, that is,
\[
\|X_{i+1}-AX_i\|<\delta\qquad (0\leq i<m).
\]
Put
$x_i=\exp(X_i)$ for $0\leq i\leq m$,
and extend this finite sequence by the true $\alpha$-orbit of $x_m$:
\[
x_{m+j}=\alpha^j(x_m)\qquad (j\geq 1).
\]
Since
$
\alpha(\exp X)=\exp(AX),
$
for any $X\in \mathfrak g$,
condition \eqref{eq:lin2gr} implies that
$(x_n)_{n\geq 0}$ is an $N$-pseudo-orbit for $\alpha$.

Therefore there exists $y\in G$ such that
$
\alpha^i(y)^{-1}x_i\in O,
$
for all $0\leq i\leq m$.
Since $O$ is symmetric, this gives
\[
\alpha^i(y)\in x_iO\subseteq CW\subseteq \exp(B_r(0))
\qquad (0\leq i\leq m),
\]
where we used \eqref{eq:nbd}. In particular, taking
$i=0$, there is $Y\in B_r(0)$ such that
$
y=\exp(Y).
$

We claim that
\begin{equation}\label{eq:claim}
\alpha^i(y)=\exp(A^iY)
\quad\text{and}\quad
A^iY\in B_r(0)
\end{equation}
for all $0\leq i\leq m$. This is clear for $i=0$. Suppose it holds for
some $i<m$. Then
\[
\alpha^{i+1}(y)
=
\alpha(\exp(A^iY))
=
\exp(A^{i+1}Y).
\]
On the other hand, we have already shown that
$\alpha^{i+1}(y)\in\exp(B_r(0))$. Hence
$\alpha^{i+1}(y)=\exp(Z)$ for some $Z\in B_r(0)$. Since
$A^iY\in B_r(0)$, we have $A^{i+1}Y\in A(B_r(0))$. By the choice of $r$,
the exponential map is injective on a set containing both $B_r(0)$ and
$A(B_r(0))$. Therefore $A^{i+1}Y=Z\in B_r(0)$. This proves the claim.

Combining the inclusion $\alpha^i(y)^{-1}x_i\in O$
 with
\eqref{eq:claim}, we get
\[
\exp(A^iY)^{-1}\exp(X_i)\in O
\qquad (0\leq i\leq m).
\]
By \eqref{eq:uni}, it follows that
\[
\|A^iY-X_i\|<\varepsilon
\qquad (0\leq i\leq m).
\]
Thus every finite $\delta$-pseudo-orbit of $A$ contained in
$B_\varrho(0)$ is $\varepsilon$-shadowed by $Y$.
This completes the proof.
\end{proof}

\begin{remark}
The argument above uses more than the fact that the exponential map is a
local diffeomorphism, or even a local uniform homeomorphism. Indeed, a finite
pseudo-orbit $X_0,\ldots,X_m$ of $A$ contained in a small neighbourhood of
$0$ gives a finite pseudo-orbit $\exp(X_0),\ldots,\exp(X_m)$ of $\alpha$ in
a small neighbourhood of $e$. The shadowing property of $\alpha$ then gives
a point $y\in G$ which shadows it. However, the local chart alone does not
imply that this shadowing point $y$, or its iterates $\alpha^i(y)$, remain
inside the same coordinate neighbourhood. This is why, in the proof, the
shadowing neighbourhood $O$ is chosen so small that
\[
\exp(\overline{B_\varrho(0)})O\subseteq \exp(B_r(0)).
\]
This forces all relevant points $\alpha^i(y)$ to stay inside the exponential
chart, allowing us to pull the shadowing relation back to the Lie algebra.
\end{remark}

If $A$ is an invertible linear map on a finite-dimensional vector space, then
$A$ is hyperbolic if and only if $A^{-1}$ is hyperbolic. Therefore
Theorem~\ref{LieChar} immediately yields the following consequence.

\begin{corollary}\label{twos}
Let $G$ be a Lie group and let $\alpha\in\operatorname{Aut}(G)$. Then the
following conditions are equivalent:
\begin{enumerate}
\item $\alpha$ has shadowing;
\item $\alpha^{-1}$ has shadowing;
\item $\alpha$ has two-sided shadowing.
\end{enumerate}
\end{corollary}

\begin{proof}
By Theorem~\ref{LieChar}, $\alpha$ has shadowing if and only if
$\di\alpha$ is hyperbolic. Since $\di(\alpha^{-1})=(\di\alpha)^{-1}$, this
is equivalent to $\alpha^{-1}$ having shadowing. The equivalence with
two-sided shadowing follows from Lemma~\ref{lem:one-two-sided}.
\end{proof}

\begin{proposition}\label{prop:Lie-Anosov-endomorphism}
Let $G$ be a compact connected Lie group and let
$\alpha\in\operatorname{End}(G)$. Assume that $\di\alpha$ is nonsingular.
Then $\alpha$ is a covering map, and the following conditions are
equivalent:
\begin{enumerate}
\item $\alpha$ has shadowing;
\item $\di\alpha$ is hyperbolic;
\item $\alpha$ is an Anosov endomorphism in the quotient-bundle sense fixed
above.
\end{enumerate}
\end{proposition}

\begin{proof}
Since $\di\alpha$ is nonsingular, $\alpha$ is a local diffeomorphism and
$\alpha(G)$ is an open subgroup of $G$. Connectedness gives
$\alpha(G)=G$. Since $G$ is compact, this surjective local diffeomorphism is
proper and hence is a finite covering map.

The equivalence of (1) and (2) is Theorem~\ref{LieChar}. Suppose that (2)
holds and write
$\mathfrak g=E^{\mathrm s}\oplus E^{\mathrm u}$ for the hyperbolic
splitting of $A=\di\alpha$. Choose an adapted inner product for which the
restriction to $E^{\mathrm s}$ contracts and the restriction to
$E^{\mathrm u}$ expands, and extend it to a left-invariant Riemannian metric
on $G$. Left translation of $E^{\mathrm s}$ gives a $D\alpha$-invariant
stable bundle. Under left translations, its quotient bundle is identified
with $E^{\mathrm u}$, on which the induced map is uniformly expanding.
Hence $\alpha$ is an Anosov endomorphism.

Conversely, if $\alpha$ is an Anosov endomorphism, evaluate its stable
bundle and expanding quotient at the fixed point $e$. The restriction of
$A$ to the stable space has spectrum inside the open unit disk, while the
induced map on the quotient has spectrum outside the closed unit disk. In a
basis adapted to the stable space, $A$ is block upper triangular, and its
spectrum is the union of the spectra of the restriction and the quotient
map. Thus $A$ is hyperbolic.
\end{proof}

\begin{remark}
The nonsingularity assumption in Proposition~\ref{prop:Lie-Anosov-endomorphism}
is exactly what places $\alpha$ in the classical theory of Anosov
endomorphisms. If $\di\alpha$ is singular, zero may occur as a stable
eigenvalue, and Theorem~\ref{LieChar} still gives shadowing although
$\alpha$ is not a local diffeomorphism.
\end{remark}

\subsection{Consequences and examples}

We now discuss positively expansive endomorphisms and expansive
automorphisms. In the Lie setting, shadowing is not an additional
assumption once the corresponding expansiveness condition holds; it follows
from the Lie algebra criterion in Theorem~\ref{LieChar}.

For Lie group endomorphisms, positive expansiveness can also be read from
the differential. For connected groups, the following criterion is the
cyclic-semigroup case of Bhattacharya's infinitesimal characterization
\cite[Theorem~A and Proposition~2.4]{Bhattacharya2004}. We include the direct
argument both for completeness and to record the extension to arbitrary,
not necessarily connected, Lie groups.

\begin{proposition}\label{PostEx}
Let $G$ be a Lie group with Lie algebra $\mathfrak g$, let
$\alpha:G\to G$ be a continuous endomorphism, and put
$A=\di\alpha:\mathfrak g\to\mathfrak g$. Then $\alpha$ is positively
expansive if and only if every eigenvalue of $A$ has modulus strictly larger
than $1$. Consequently, if $G$ is connected, then $\alpha$ is surjective;
if $G$ is simply connected, then $\alpha$ is bijective.
\end{proposition}

\begin{proof}
Assume first that $\alpha$ is positively expansive. Let $V$ be a
neighbourhood of $e$ such that
\[
\bigcap_{n\geq 0}\alpha^{-n}(V)=\{e\}.
\]
Shrinking $V$ if necessary, we may choose a neighbourhood $U$ of $0$ in
$\mathfrak g$ such that $\exp|_U:U\to V$ is a homeomorphism.

Suppose that $A$ has an eigenvalue of modulus at most $1$. Then there is a
non-zero vector $X\in\mathfrak g$ such that $(A^nX)_{n\geq 0}$ is bounded.
 Replacing $X$ by $tX$ for a
sufficiently small $t>0$, we may assume that $A^nX\in U$ for every
$n\geq 0$. Hence
\[
\alpha^n(\exp X)=\exp(A^nX)\in V
\qquad (n\geq 0).
\]
Thus $\exp X\in \bigcap_{n\geq 0}\alpha^{-n}(V)$. Since $X\neq 0$ and
$\exp$ is injective on $U$, this contradicts the choice of $V$. Therefore
all eigenvalues of $A$ have modulus strictly larger than $1$.

Conversely, assume that all eigenvalues of $A$ have modulus strictly larger
than $1$. Choose an equivalent norm on $\mathfrak g$ and a number
$\lambda>1$ such that $\|AX\|\geq \lambda\|X\|$ for all
$X\in\mathfrak g$. Choose $r>0$ such that $\exp$ is injective on $B_r(0)$.
Let $M>1$ satisfy $\|AX\|\leq M\|X\|$ for all $X\in\mathfrak g$, and set
\[
V=\exp(B_{r/M}(0)).
\]
We show that $V$ witnesses positive expansiveness.

Let $g\neq e$. If $g\notin V$, there is nothing to prove. Otherwise write
$g=\exp X$ with $X\in B_{r/M}(0)$. Since $\exp$ is injective on $B_r(0)$,
we have $X\neq 0$. As $\|A^kX\|\to\infty$, let $k\geq 0$ be the least
integer such that $A^kX\notin B_{r/M}(0)$. Then $k\geq 1$ and
$A^{k-1}X\in B_{r/M}(0)$, so $\|A^kX\|<r$. Hence
$A^kX\in B_r(0)\setminus B_{r/M}(0)$. Therefore
\[
\alpha^k(g)=\exp(A^kX)\notin V,
\]
because $\exp$ is injective on $B_r(0)$. Thus no non-trivial element remains
in $V$ under all forward iterates, and
$\bigcap_{n\geq 0}\alpha^{-n}(V)=\{e\}$. Hence $\alpha$ is positively
expansive.

The final assertions follow as follows. The condition on the eigenvalues
implies that $A$ is invertible, and hence $\alpha$ is a local diffeomorphism
at $e$. Therefore $\alpha(G)$ is an open subgroup of $G$. If $G$ is
connected, every open subgroup is all of $G$, so $\alpha$ is surjective. If
$G$ is simply connected, then this surjective local diffeomorphism is a
covering map of $G$ onto itself; since $G$ is simply connected, the covering
is trivial. Hence $\alpha$ is injective, and therefore bijective.
\end{proof}

\begin{lemma}[{\cite[Part~I, \S4, Exercise~21(b)]{BourbakiLie}}]\label{Bou}
	Let $\mathfrak g$ be a finite-dimensional real Lie algebra. If
	$\mathfrak g$ admits a hyperbolic automorphism, then $\mathfrak g$ is
	nilpotent.
\end{lemma}

Combining the preceding proposition with Theorem~\ref{LieChar}, we obtain the promised relation between positive expansiveness and topological expansion.

\begin{theorem}\label{thm:positive-expansive-shadowing}
Let $G$ be a Lie group, and let $\alpha:G\to G$ be a continuous endomorphism.
If $\alpha$ is positively expansive, then $\alpha$ has the shadowing property.
Consequently, $\alpha$ is topologically expanding if and only if it is
positively expansive. Moreover, if $G$ admits a positively expansive
endomorphism, then the identity component $G_0$ is nilpotent.
\end{theorem}

\begin{proof}
Put $A=\di\alpha:\mathfrak g\to\mathfrak g$. If $\alpha$ is positively
expansive, then by Proposition~\ref{PostEx}, every eigenvalue of $A$ has
modulus strictly larger than $1$. In particular, $A$ is hyperbolic. Hence,
by Theorem~\ref{LieChar}, $\alpha$ has the shadowing property.

The equivalence with being topologically expanding follows immediately from
the definition: a topologically expanding endomorphism is a positively
expansive endomorphism with shadowing, and we have just shown that positive
expansiveness already implies shadowing.

It remains to prove the last assertion. Suppose that $G$ admits a positively
expansive endomorphism $\alpha$. By Proposition~\ref{PostEx}, $\di\alpha$
is an automorphism all of whose eigenvalues have modulus strictly larger
than $1$; in particular, it is a hyperbolic automorphism. Hence, by
Lemma~\ref{Bou}, the Lie algebra $\mathfrak g$ is nilpotent. Therefore the
identity component $G_0$ is nilpotent.
\end{proof}

\begin{remark}
If $G$ is compact and connected, Proposition~\ref{PostEx} shows that a
positively expansive endomorphism is a smooth expanding map in the sense of
Shub \cite{Shub1969}: after choosing an adapted left-invariant metric, its
differential expands every tangent vector uniformly. In this case $G$ is a
compact connected nilpotent Lie group and hence a torus. Thus
Theorem~\ref{thm:positive-expansive-shadowing} recovers the familiar fact that
expanding toral endomorphisms have shadowing, while also treating
non-compact Lie groups. More generally, automatic shadowing for positively
expansive maps in the compact Lie setting follows from Sakai
\cite[Theorem~1]{Sakai2003}; see also \cite[Lemma~9]{Sakai2010}. The torus
conclusion in the compact connected case also follows from Eisenberg's
classical result
\cite{EisenbergPositive1966}, since Proposition~\ref{PostEx} gives
surjectivity.
\end{remark}

We next turn from positively expansive endomorphisms to expansive
automorphisms. We shall use the following known characterization.

\begin{lemma}\label{Exp}
Let $G$ be a Lie group and let $\alpha\in\operatorname{Aut}(G)$. Then
$\alpha$ is expansive if and only if $\di\alpha$ is hyperbolic.
\end{lemma}

\begin{proof}
Let $G_0$ be the identity component and put
$\alpha_0=\alpha|_{G_0}$. The subgroup $G_0$ is characteristic and open in
$G$. We first observe that $\alpha$ is expansive on $G$ if and only if
$\alpha_0$ is expansive on $G_0$. One implication follows by intersecting
an expansive neighbourhood in $G$ with $G_0$. Conversely, if $V\subseteq
G_0$ is an expansive neighbourhood for $\alpha_0$, then $V$ is also a
neighbourhood of $e$ in $G$, and
\[
\bigcap_{n\in\mathbb Z}\alpha^{-n}(V)
=
\bigcap_{n\in\mathbb Z}\alpha_0^{-n}(V)
=\{e\};
\]
membership at time $n=0$ already forces the point to lie in $G_0$.

If $G_0=\{e\}$, then $G$ is discrete, every automorphism is expansive,
and $\di\alpha$ acts hyperbolically on the zero vector space. If
$G_0\neq\{e\}$, Shah's characterization
\cite[Proposition~2.3]{Shah2020}, applied to $\alpha_0$, shows that
$\alpha_0$ is expansive if and only if
$\di\alpha_0=\di\alpha$ is hyperbolic. The assertion follows.
\end{proof}

Together with Theorem~\ref{LieChar} and Corollary~\ref{twos}, this gives the following characterization of topologically Anosov automorphisms of Lie groups.

\begin{theorem}\label{thm:Lie-topological-Anosov}
Let $G$ be a Lie group and let $\alpha\in\operatorname{Aut}(G)$. Then the
following conditions are equivalent:
\begin{enumerate}
\item $\alpha$ is topologically Anosov;
\item $\alpha$ is expansive;
\item $\alpha$ has the shadowing property;
\item $\alpha$ has the two-sided shadowing property.
\end{enumerate}
Moreover, these conditions hold only when the identity component of $G$ is
nilpotent.
\end{theorem}

\begin{proof}
By Lemma~\ref{Exp}, $\alpha$ is expansive if and only if $\di\alpha$ is
hyperbolic. By Theorem~\ref{LieChar}, this is equivalent to $\alpha$ having
shadowing. Corollary~\ref{twos} gives the equivalence with two-sided
shadowing, and the equivalence with being topologically Anosov follows from
the definition.

Finally, $\di\alpha$ is a hyperbolic automorphism of the Lie algebra
$\mathfrak g$ of $G$. By Lemma~\ref{Bou}, $\mathfrak g$ is nilpotent, and
therefore the identity component $G_0$ is nilpotent.
\end{proof}

The following standard observation identifies the classical smooth
notion of Anosov dynamics for Lie group automorphisms with hyperbolicity
of the differential; compare \cite{Morimoto1981,Bhattacharya2004,Shah2020}.

\begin{proposition}\label{prop:classical-Anosov-Lie}
Let $G$ be a compact Lie group and let
$\alpha\in\operatorname{Aut}(G)$. Then $\alpha$ is an Anosov
diffeomorphism in the classical smooth sense if and only if
$\di\alpha$ is hyperbolic.
\end{proposition}

\begin{proof}
Suppose that $A=\di\alpha$ is hyperbolic, and write
$\mathfrak g=E^{\mathrm s}\oplus E^{\mathrm u}$ for its hyperbolic
splitting. Choose an adapted norm on $\mathfrak g$, and let the
corresponding left-invariant Riemannian metric on $G$ be fixed. Define
\[
E_x^{\mathrm s}=\di L_x(E^{\mathrm s}),
\qquad
E_x^{\mathrm u}=\di L_x(E^{\mathrm u}).
\]
Since $\alpha\circ L_x=L_{\alpha(x)}\circ\alpha$, the splitting
$TG=E^{\mathrm s}\oplus E^{\mathrm u}$ is $D\alpha$-invariant, and the
usual uniform exponential estimates follow directly from those for $A$.
Thus $\alpha$ is an Anosov diffeomorphism.

Conversely, if $\alpha$ is Anosov, then $e$ is a fixed point and the
Anosov splitting at $e$ shows that $D\alpha_e=\di\alpha$ is hyperbolic.
\end{proof}

\begin{corollary}\label{cor:compact-Lie-Anosov}
Let $G$ be a compact Lie group and let
$\alpha\in\operatorname{Aut}(G)$. Then the following are equivalent:
\begin{enumerate}
\item $\alpha$ is an Anosov diffeomorphism;
\item $\alpha$ is topologically Anosov;
\item $\alpha$ is expansive;
\item $\alpha$ has shadowing;
\item $\alpha$ has two-sided shadowing.
\end{enumerate}
\end{corollary}

\begin{proof}
Combine Proposition~\ref{prop:classical-Anosov-Lie} with
Theorem~\ref{thm:Lie-topological-Anosov}.
\end{proof}

\begin{remark}
For compact connected Lie groups, Corollary~\ref{cor:compact-Lie-Anosov}
is essentially classical. Morimoto treated Euclidean and toral
automorphisms and obtained related structural restrictions for
automorphisms of compact connected Lie groups \cite{Morimoto1981}; Aoki's
solenoidal result is another closely related compact connected precedent
\cite[Theorem~2]{AokiSolenoidal1983}. Theorem~\ref{LieChar} is broader in two
directions: the ambient Lie group need not be compact, and the map need not
be invertible or even locally invertible.
\end{remark}

We shall also need a simple permanence property of shadowing under taking powers. Although elementary, it is useful for treating nilpotent endomorphisms.

\begin{lemma}\label{powern}
	Let $G$ be a topological group and let $\alpha\in\operatorname{End}(G)$.
	Then, for every $n\geq 1$, $\alpha$ has shadowing if and only if
	$\alpha^n$ has shadowing. In particular, $\alpha$ has shadowing if and only
	if there exists $n\geq 1$ such that $\alpha^n$ has shadowing.
\end{lemma}

\begin{proof}
	First assume that $\alpha$ has shadowing, and fix $n\geq 1$. Let $U$ be a
	neighbourhood of $e$. Choose a neighbourhood $V$ of $e$ such that every
	$V$-pseudo-orbit for $\alpha$ is $U$-shadowed by an orbit of $\alpha$.
	
	Let $(y_k)_{k\geq 0}$ be a $V$-pseudo-orbit for $\alpha^n$, that is,
	\[
	\alpha^n(y_k)^{-1}y_{k+1}\in V
	\qquad (k\geq 0).
	\]
	Define a sequence $(x_i)_{i\geq 0}$ by
	\[
	x_{kn+j}=\alpha^j(y_k),
	\qquad k\geq 0,\ 0\leq j<n.
	\]
	Then $(x_i)$ is a $V$-pseudo-orbit for $\alpha$: all intermediate steps are
	exact, and the only possible error occurs from $x_{kn+n-1}$ to $x_{kn+n}$,
	where it is precisely
	\[
	\alpha^n(y_k)^{-1}y_{k+1}\in V.
	\]
	Hence there exists $z\in G$ such that
	\[
	\alpha^i(z)^{-1}x_i\in U
	\qquad (i\geq 0).
	\]
	Taking $i=kn$, we get
	\[
	(\alpha^n)^k(z)^{-1}y_k
	=
	\alpha^{kn}(z)^{-1}x_{kn}
	\in U
	\qquad (k\geq 0).
	\]
	Thus $(y_k)$ is $U$-shadowed by the $\alpha^n$-orbit of $z$. Therefore
	$\alpha^n$ has shadowing.
	
	Conversely, assume that $\alpha^n$ has shadowing. Let $U$ be a
	neighbourhood of $e$. Choose neighbourhoods $W,H$ of $e$ such that
	\[
	WH\subseteq U.
	\]
	By continuity of the maps $\alpha^j$, $0\leq j<n$, choose a neighbourhood
	$R$ of $e$ such that
	\[
	\alpha^j(R)\subseteq W
	\qquad (0\leq j<n).
	\]
	Let $S$ be a neighbourhood of $e$ such that every $S$-pseudo-orbit for
	$\alpha^n$ is $R$-shadowed by an orbit of $\alpha^n$.
	
	For a neighbourhood $V$ of $e$, put $P_0(V)=\{e\}$ and
	\[
	P_j(V)=\alpha^{j-1}(V)\alpha^{j-2}(V)\cdots\alpha(V)V
	\qquad(j\geq 1).
	\]
	Choose a neighbourhood $V$ of $e$ so small that
	\[
	P_n(V)\subseteq H\cap S.
	\]
	This is possible by continuity of multiplication and of the finitely many
	maps $\alpha^j$. Since $e\in V$, it also follows that
	\[
	P_j(V)\subseteq P_n(V)\subseteq H\cap S
	\qquad(0\leq j\leq n).
	\]
	
	Let $(x_i)_{i\geq 0}$ be a $V$-pseudo-orbit for $\alpha$. Put
	$e_i=\alpha(x_i)^{-1}x_{i+1}\in V$. For $j\geq 1$ we have
	\[
	x_{i+j}
	=
	\alpha^j(x_i)\alpha^{j-1}(e_i)\alpha^{j-2}(e_{i+1})\cdots e_{i+j-1}.
	\]
	In particular, the subsequence $(x_{kn})_{k\geq 0}$ is an $S$-pseudo-orbit
	for $\alpha^n$, since
	\[
	\alpha^n(x_{kn})^{-1}x_{(k+1)n}
	\in
	P_n(V)\subseteq S.
	\]
	Hence there exists $z\in G$ such that
	\[
	\alpha^{kn}(z)^{-1}x_{kn}\in R
	\qquad (k\geq 0).
	\]
	
	Now fix $i\geq 0$ and write $i=kn+j$ with $0\leq j<n$. Then
	\[
	\alpha^i(z)^{-1}x_i
	=
	\alpha^j\bigl(\alpha^{kn}(z)^{-1}x_{kn}\bigr)
	\cdot
	\alpha^j(x_{kn})^{-1}x_{kn+j}.
	\]
	The first factor lies in $\alpha^j(R)\subseteq W$. By the preceding
	product formula, the second factor belongs to $P_j(V)\subseteq H$ (and it
	equals $e$ when $j=0$). Hence
	\[
	\alpha^i(z)^{-1}x_i\in WH\subseteq U.
	\]
	Thus $(x_i)$ is $U$-shadowed by the $\alpha$-orbit of $z$. Therefore
	$\alpha$ has shadowing.
\end{proof}

We now apply the preceding results to semisimple Lie groups. Let us recall
that an endomorphism $\alpha$ of a topological group $G$ is called
\emph{nilpotent} if $\alpha^n$ is the trivial endomorphism for some
$n\geq 1$. If $G$ is a connected Lie group, this is equivalent to the
nilpotency of $\di\alpha$. Indeed,
$\di(\alpha^n)=(\di\alpha)^n$, and if $\di(\alpha^n)=0$, then
$\alpha^n(G)$ is a connected Lie subgroup with zero Lie algebra, hence is
trivial.

\begin{proposition}\label{semisim}
	Let $G$ be a connected semisimple Lie group, and let
	$\alpha\in\operatorname{End}(G)$. Then $\alpha$ has shadowing if and only if
	$\alpha$ is nilpotent.
\end{proposition}

\begin{proof}
	Put $A=\di\alpha:\mathfrak g\to\mathfrak g$.
	
	Assume first that $\alpha$ has shadowing. By Theorem~\ref{LieChar}, $A$ is
	hyperbolic. Since $\mathfrak g$ is semisimple, we claim that $A$ must be
	nilpotent. Indeed, choose $m\geq 1$ such that
	\[
	A^m(\mathfrak g)=A^{m+1}(\mathfrak g).
	\]

Set $\mathfrak h=A^m(\mathfrak g)$. Then $\mathfrak h$ is semisimple,
because it is isomorphic to the quotient $\mathfrak g/\ker A^m$. Moreover,
$A(\mathfrak h)=\mathfrak h$, so $A|_{\mathfrak h}$ is an automorphism of
$\mathfrak h$. Since $\mathfrak h$ is $A$-invariant, the spectrum of
$A|_{\mathfrak h}$ is contained in the spectrum of $A$. Hence
$A|_{\mathfrak h}$ is hyperbolic. If $\mathfrak h\neq 0$, then
Lemma~\ref{Bou} implies that $\mathfrak h$ is nilpotent, contradicting
semisimplicity. Therefore $\mathfrak h=0$, and hence $A$ is nilpotent.

	Since $G$ is connected, nilpotency of $A=\di\alpha$ is equivalent to
	nilpotency of $\alpha$. Therefore $\alpha$ is nilpotent.
	
	Conversely, suppose that $\alpha$ is nilpotent. Then $\alpha^n$ is the
	trivial endomorphism for some $n\geq 1$. The trivial endomorphism has
	shadowing. Hence, by Lemma~\ref{powern}, $\alpha$ has shadowing.
\end{proof}

This last conclusion is specific to the Lie setting. It cannot be extended to all connected compact groups. Recall that a connected compact group $G$ is called \emph{semisimple} if $[G,G]=G$, or equivalently, if $G$ is a quotient of a possibly infinite product of simple connected compact Lie groups by a totally disconnected closed normal subgroup \cite{HM}. The following example shows that a connected semisimple compact group may admit a shadowing automorphism.

\begin{example}\label{example}
	Let $S$ be a non-abelian simple connected compact Lie group, and let
	\[
	G=S^{\mathbb Z}.
	\]
	Then $G$ is a connected semisimple compact group.
	Let $\sigma:G\to G$ be the left shift, given by
	\[
	\sigma((g_i)_{i\in\mathbb Z})=(g_{i+1})_{i\in\mathbb Z}.
	\]
	Then $\sigma$ is a continuous automorphism of $G$. We show that $\sigma$
	has shadowing.
	
	Let $U$ be a neighbourhood of $e$ in $G$. Since $G$ has the product
	topology, there exist an integer $N\geq 1$ and a neighbourhood $W$ of $e$
	in $S$ such that
	\[
	\{g\in G:g_i\in W\text{ for all }|i|\leq N\}\subseteq U.
	\]
	Choose a neighbourhood $L$ of $e$ in $S$ such that every product of at most
	$N$ elements of $L\cup L^{-1}$ lies in $W$. Define a neighbourhood $V$ of
	$e$ in $G$ by
	\[
	V=\{g\in G:g_i\in L\text{ for all }|i|\leq N\}.
	\]
	
	Let $(x^n)_{n\geq 0}$ be a $V$-pseudo-orbit for $\sigma$. Write
	$x^n=(x^n_i)_{i\in\mathbb Z}$. The condition means that, for every
	$n\geq 0$ and every $|i|\leq N$,
	\[
	(x^n_{i+1})^{-1}x^{n+1}_i\in L.
	\]
	For later use, put
	\[
	\ell^n_i=(x^n_{i+1})^{-1}x^{n+1}_i\in L
	\qquad(n\geq0,\ |i|\leq N).
	\]
	Define $y=(y_i)_{i\in\mathbb Z}\in G$ by
	\[
	y_i=
	\begin{cases}
		x^i_0, & i\geq 0,\\
		x^0_i, & i<0.
	\end{cases}
	\]
	We claim that the orbit of $y$ $U$-shadows $(x^n)_{n\geq 0}$.
	
	Fix $n\geq 0$ and $|i|\leq N$. We need to show that the $i$-th coordinate
	of $\sigma^n(y)^{-1}x^n$ lies in $W$, that is,
	\[
	y_{n+i}^{-1}x^n_i\in W.
	\]
	If $i=0$, then $y_n=x^n_0$ and the required element is $e$. If $i>0$,
	repeated use of
	$x^{m+1}_j=x^m_{j+1}\ell^m_j$ gives
	\[
	x^{n+i}_0
	=x^n_i\ell^n_{i-1}\ell^{n+1}_{i-2}\cdots\ell^{n+i-1}_0.
	\]
	Thus $y_{n+i}^{-1}x^n_i$ is an ordered product of $i\leq N$ elements
	of $L^{-1}$, and hence belongs to $W$.

	Suppose next that $i<0$ and $n+i\geq0$. Then
	\[
	x^n_i
	=x^{n+i}_0\ell^{n+i}_{-1}\ell^{n+i+1}_{-2}\cdots
	\ell^{n-1}_i.
	\]
	Consequently, $y_{n+i}^{-1}x^n_i$ is an ordered product of
	$-i\leq N$ elements of $L$, and again belongs to $W$.

	If $n+i<0$, then $y_{n+i}=x^0_{n+i}$. In this case necessarily
	$0\leq n\leq |i|\leq N$. If $n=0$, then $y_i=x_i^0$ and the required
	element is $e$. If $n\geq1$, then
	\[
	x^n_i
	=x^0_{n+i}\ell^0_{n+i-1}\ell^1_{n+i-2}\cdots\ell^{n-1}_i.
	\]
	Therefore $y_{n+i}^{-1}x^n_i$ is an ordered product of $n\leq N$
	elements of $L$, and belongs to $W$ also in this case.
	
	Therefore $\sigma^n(y)^{-1}x^n\in U$ for every $n\geq 0$. Hence $\sigma$
	has shadowing.
\end{example}

\section{The Totally Disconnected Case}\label{tdlc}

We now turn to totally disconnected locally compact groups. The starting
point of the theory is van Dantzig's theorem, which asserts that every
totally disconnected locally compact group has a neighbourhood basis at the
identity consisting of compact open subgroups. Thus, in the totally
disconnected case, compact open subgroups play a role analogous to small
neighbourhoods of the identity in Lie theory.

For compact metrizable groups, the automatic shadowing of automorphisms is
classical. Aoki proved that every continuous automorphism of a compact
totally disconnected metrizable group has the pseudo-orbit tracing
property \cite[Application~2, pp.~170--172]{AokiSplitting1985}; see also
\cite[Theorem~D]{AokiDense1985}. Two endomorphism subclasses are also already
implicit in the literature. Reid's invariant open-subgroup basis for
type~(F) profinite groups gives shadowing by the elementary error recursion
\cite[Definition~3.1 and Proposition~3.2]{Reid2014}, while Wibmer's symbolic
model for positively expansive profinite endomorphisms reduces that case to
a one-sided group shift of finite type
\cite[Proposition~2.1 and Lemma~3.3]{Wibmer2022}. A standard
countable-quotient argument likewise removes metrizability in Aoki's compact
automorphism setting. The aim of this section is to treat arbitrary totally
disconnected locally compact groups and arbitrary continuous endomorphisms
in one statement.

For a long time, however, van Dantzig's theorem remained one of the few
general structural tools available for arbitrary totally disconnected
locally compact groups. A major breakthrough was Willis' theory \cite{Willis1994,Willis2001} of the scale
function and tidy subgroups for automorphisms. This theory gives a
quantitative measure of how an automorphism moves compact open subgroups,
and it provides canonical compact open subgroups adapted to the dynamics of
the automorphism. It has become one of the basic tools in the modern
structure theory of totally disconnected locally compact groups, and it is
also inherently dynamical in nature.

Willis' theory was later extended from automorphisms to endomorphisms. We
shall use only the tidy-above part of this theory. Let $G$ be a totally
disconnected locally compact group, let $\alpha\in\operatorname{End}(G)$,
and let $U$ be a compact open subgroup of $G$. Define
\[
U_-=\{u\in U:\alpha^n(u)\in U\text{ for all }n\geq 0\}=\bigcap_{n\geq 0} \alpha^{-n}(U),
\]
and
\[
\begin{aligned}
U_+=\{u\in U:\ &\text{there exists }(u_n)_{n\geq 0}\subseteq U
\text{ such that }u_0=u,\\
&\text{and }\alpha(u_{n+1})=u_n\text{ for all }n\geq 0\}.
\end{aligned}
\]
The subgroup $U$ is said to be \emph{tidy above} for $\alpha$ if
\[
U=U_+U_-.
\]
By the tidying procedure for endomorphisms, compact open subgroups tidy
above for $\alpha$ form a neighbourhood basis at the identity.

\begin{lemma}[{\cite[Proposition~3.9]{WillisEndomorphisms}}]\label{tabase}
Let $G$ be a totally disconnected locally compact group and
$\alpha\in\operatorname{End}(G)$. Then
for every compact open subgroup $U$ of $G$, there exists a compact open
subgroup $V\leq U$ which is tidy above for $\alpha$.
\end{lemma}

We shall show that this decomposition is exactly what is needed to trace
pseudo-orbits. The $U_-$ part controls forward errors, while the $U_+$ part
allows us to correct them by choosing suitable regressive trajectories.

\begin{theorem}\label{Th:tdlc}
Let $G$ be a totally disconnected locally compact group, and let
$\alpha:G\to G$ be a continuous endomorphism. Then $\alpha$ has the
shadowing property.
\end{theorem}

\begin{proof} 
Let $U\leq G$ be compact open. By Lemma~\ref{tabase}, there is a compact
open subgroup $V\leq U$ which is
tidy above for $\alpha$. Recall that, in this setting,
$V=V_+V_-$. Since $V_+$ and $V_-$ are subgroups, the equality $V=V_+V_-$ also gives
$V=V_-V_+$.

Let $(x_n)_{n\geq 0}$ be a $V$-pseudo-orbit, and put
$e_n=\alpha(x_n)^{-1}x_{n+1}\in V$. We shall prove that, for every
$m\geq 0$, there exists $y_m\in G$ such that
$\alpha^i(y_m)\in x_iV$ for all $0\leq i\leq m$. More precisely, we prove
by induction that $y_m$ can be chosen so that
$\alpha^m(y_m)=x_mb_m$ for some $b_m\in V_-$.

For $m=0$, take $y_0=x_0$ and $b_0=e$. Suppose that $y_m$ and
$b_m\in V_-$ have been chosen. Since $x_{m+1}=\alpha(x_m)e_m$, we have
\[
\alpha^{m+1}(y_m)=\alpha(x_m)\alpha(b_m)
=x_{m+1}e_m^{-1}\alpha(b_m).
\]
Set $u=e_m^{-1}\alpha(b_m)$. Then $u\in V$, since $e_m\in V$ and
\[
\alpha(b_m)\in \alpha(V_-)\leq V_-\leq V.
\] 
Write $u=ba$ with $b\in V_-$ and
$a\in V_+$. Choose
$a_0,a_1,\ldots,a_{m+1}$ in $V_+$ with $a_0=a$ and
$\alpha(a_{j+1})=a_j$ for $0\leq j\leq m$. Put 
$y_{m+1}=y_ma_{m+1}^{-1}$.

For $0\leq i\leq m$, we have $\alpha^i(a_{m+1})\in V$, and hence
$\alpha^i(y_{m+1})\in \alpha^i(y_m)V=x_i V$. Moreover,
\[
\alpha^{m+1}(y_{m+1})
=\alpha^{m+1}(y_m)\alpha^{m+1}(a_{m+1})^{-1}
=x_{m+1}ua^{-1}=x_{m+1}b,
\]
with $b\in V_-$. This completes the induction.

Now define
$K_m=\{y\in x_0V:\alpha^i(y)\in x_iV\text{ for }0\leq i\leq m\}$.
The preceding paragraph shows that each $K_m$ is non-empty. The sets $K_m$
are closed subsets of the compact set $x_0V$, and they form a decreasing
sequence. Hence $\bigcap_{m\geq 0}K_m\neq\emptyset$. Choose
$y\in\bigcap_{m\geq 0}K_m$. Then $\alpha^n(y)\in x_nV$ for every
$n\geq 0$, so $\alpha^n(y)^{-1}x_n\in V\leq U$ for every $n\geq 0$.
Thus $(x_n)_{n\geq 0}$ is $U$-shadowed by the forward orbit of $y$.

Since compact open subgroups form a neighbourhood basis at the identity,
the preceding argument proves shadowing for every neighbourhood of $e$.
\end{proof}

\begin{remark}\label{rem:Aoki-automatic-shadowing}
When $G$ is compact and metrizable and $\alpha$ is an automorphism, the
preceding theorem is due to Aoki
\cite[Application~2, pp.~170--172]{AokiSplitting1985}.
Theorem~\ref{Th:tdlc}
passes beyond the compact automorphism setting by allowing $G$ to be
non-compact and $\alpha$ to be non-invertible, without imposing
metrizability. Reid's type~(F) result and
Wibmer's positively expansive profinite result cover the special
endomorphism classes described above. We have not found the full generality
of Theorem~\ref{Th:tdlc} stated in the earlier literature.
\end{remark}

\begin{corollary}\label{cor:tdlc-expanding-Anosov}
Let $G$ be a totally disconnected locally compact group.
\begin{enumerate}
\item If $\alpha:G\to G$ is a continuous endomorphism, then $\alpha$ is
topologically expanding if and only if it is positively expansive.
\item If $\alpha\in\operatorname{Aut}(G)$, then $\alpha$ is topologically
Anosov if and only if it is expansive.
\end{enumerate}
\end{corollary}

\begin{proof}
By Theorem~\ref{Th:tdlc}, every continuous endomorphism of a totally
disconnected locally compact group has shadowing. Hence, for endomorphisms,
positive expansiveness is equivalent to positive expansiveness together
with shadowing, that is, to being topologically expanding.

Similarly, every automorphism $\alpha$ of a totally disconnected locally
compact group has shadowing by Theorem~\ref{Th:tdlc}, and hence has
two-sided shadowing by Lemma~\ref{lem:one-two-sided}. Therefore, for
automorphisms, expansiveness is equivalent to expansiveness together with
two-sided shadowing, that is, to being topologically Anosov.
\end{proof}

We conclude by combining Theorem~\ref{Th:tdlc} with Aoki's two
propositions on dense orbits. In the metrizable setting, Aoki proved that a
topologically mixing automorphism with the shadowing property
can act only on a compact group \cite[Proposition~3]{AokiDense1985}. He also
proved that an automorphism with a dense orbit is topologically mixing and
has the shadowing property \cite[Proposition~4]{AokiDense1985}. Thus,
in that setting, the weaker hypothesis of a dense orbit already implies
compactness. Previts and Wu later revisited Aoki's compactness theorem and
gave a proof, under its Haar-ergodicity formulation, based on Willis' theory
\cite{PrevitsWu2003}. Accordingly, the compactness conclusion itself should
not be regarded as new here. The proof below packages the compact quotient
reduction together with the automatic-shadowing theorem used in that
reduction.

Theorem~\ref{Th:tdlc} shows in addition that the shadowing assertion in
Aoki's Proposition~4 is automatic. Consequently, the part of Aoki's metric
argument devoted to establishing the pseudo-orbit tracing property can be
replaced by Theorem~\ref{Th:tdlc}; only the argument from a dense orbit to
topological mixing remains. The following formulation also removes
metrizability.

\begin{corollary}\label{cor:dense-orbit-compact}
Let $G$ be a totally disconnected locally compact group, and let
$\alpha\in\operatorname{Aut}(G)$. If $\alpha$ has a dense orbit, then $G$
is compact. Moreover, $\alpha$ is topologically mixing.
\end{corollary}

\begin{proof}
Let $x\in G$ have a dense $\alpha$-orbit. Then $G$ is separable and
$\sigma$-compact. Indeed, if $U$ is a compact open subgroup, then the
countable set $\{\alpha^n(x):n\in\mathbb Z\}$ is dense and its left
$U$-translates cover $G$.

By the Kakutani--Kodaira theorem, there is a compact normal subgroup $K_0$
of $G$ such that $G/K_0$ is metrizable. Put
\[
K=\bigcap_{n\in\mathbb Z}\alpha^n(K_0).
\]
Then $K$ is compact, normal and $\alpha$-invariant. Moreover, $G/K$ embeds
into the countable product
\[
\prod_{n\in\mathbb Z}G/\alpha^n(K_0),
\]
so $G/K$ is metrizable. Let $\bar\alpha$ be the induced automorphism of
$G/K$. The orbit of $xK$ is dense in $G/K$.

By Aoki's Proposition~4, $\bar\alpha$ is topologically mixing. By
Theorem~\ref{Th:tdlc} and Lemma~\ref{lem:one-two-sided}, it also has the
pseudo-orbit tracing property with respect to a compatible left-invariant
metric. Aoki's Proposition~3 therefore implies that $G/K$ is compact.
Since $K$ is compact, $G$ is compact.

Finally, a compact group automorphism with a dense orbit is ergodic with
respect to normalized Haar measure \cite{RajagopalanSchreiber1970}, and an
ergodic automorphism of a compact group is mixing
\cite{Yuzvinskii1967}. Since every non-empty open subset of $G$ has
positive Haar measure, measure-theoretic mixing implies topological
mixing. Hence $\alpha$ is topologically mixing.
\end{proof}

\section*{Acknowledgements}

\noindent The results of this paper were first presented at a seminar in
May. The author is grateful to Professor Wei He for his helpful suggestions.

\noindent\textbf{Financial support.}
This work was supported by NSFC Grants 12301089 and 12271258.

\noindent\textbf{Use of artificial intelligence tools.}
During the preparation of the initial draft, the author used OpenAI's
GPT-5.5 as a writing aid and to carry out the principal calculations in the
proof of Lemma~\ref{lem:local-error-equation}. Example~\ref{example} was
initially suggested by GPT-5.5. After completing the initial draft, the
author used OpenAI's GPT-5.6-sol to assist
with editorial polishing, checking the mathematical arguments, and
searching the prior literature. The literature search assisted by
GPT-5.6-sol led to the identification of the earlier result discussed in
Remark~\ref{rem:Aoki-automatic-shadowing}. The author carefully reviewed
the entire manuscript and independently verified all calculations,
arguments, citations and bibliographic information. The author made all
final decisions concerning the manuscript and accepts full responsibility
for the accuracy, originality and content of the article.

\end{document}